\documentclass[letterpaper, 10 pt, conference]{ieeeconf}  
\IEEEoverridecommandlockouts                              
\usepackage{graphicx} 

\usepackage{amsfonts,amsmath,amssymb} 

\usepackage{array}
\usepackage{subcaption}

\graphicspath{{Figures/}}

\usepackage[utf8]{inputenc}
\usepackage[T1]{fontenc}
\usepackage[english]{babel} 
\usepackage{cite}

\usepackage{amsmath}
\usepackage{amssymb}
\usepackage{eucal}

\usepackage{tikz}
\usepackage{pgfplots}
\pgfplotsset{compat=newest}
\pgfplotsset{plot coordinates/math parser=false}
\tikzstyle{dotted} = [line width = 0.5pt, dash pattern = on \pgflinewidth off 2\pgflinewidth]
\pgfplotsset{%
	tick label style={font=\scriptsize,%
		/pgf/number format/precision=3,%
		/pgf/number format/std=-3:3},%
	legend style={font=\small},%
	label style={font=\small},%
	every axis post/.append style={label style={font=\small},%
		axis line style=-,%
		scaled ticks=false}%
}

\newlength{\figurewidth}
\newlength{\figureheight}

\newtheorem{thm}{Theorem}
\newtheorem{crl}{Corollary}
\newtheorem{prop}{Proposition}
\newtheorem{lem}{Lemma}
\newtheorem{rmk}{Remark}

\providecommand{\abs}[1]{\left\lvert#1\right\rvert}
\providecommand{\norm}[1]{\left\lVert#1\right\rVert}

\newcommand{\be}{\begin{equation}}
\newcommand{\ee}{\end{equation}}
\newcommand{\beno}{\begin{equation*}}
\newcommand{\eeno}{\end{equation*}}
\newcommand{\barCl}{\begin{IEEEeqnarray}{rCl}}
\newcommand{\earCl}{\end{IEEEeqnarray}}
\newcommand{\barClno}{\begin{IEEEeqnarray*}{rCl}}
\newcommand{\earClno}{\end{IEEEeqnarray*}}
\newcommand{\ans}{\IEEEeqnarraynumspace}
\newcommand{\ase}{\IEEEyessubnumber}

\makeatletter
\newcommand{\my@opt@exp}[2]{%
	\ifx\relax#2\relax%
		#1%
	\else%
		{#1}^{#2}%
	\fi%
}
\newcommand{\my@opt@sub}[2]{%
	\ifx\relax#2\relax%
	#1%
	\else%
	{#1}_{#2}%
	\fi%
}
\newcommand{\Reals}[1][]{\my@opt@exp{\mathbb{R}}{#1}}
\newcommand{\eps}[1][]{\my@opt@exp{\varepsilon}{#1}}
\newcommand{\y}[1][]{\my@opt@exp{y}{#1}}
\newcommand{\ff}[1][]{\my@opt@exp{f}{#1}}			
\newcommand{\favg}[1][]{\my@opt@exp{g}{#1}}			
\newcommand{\fa}{f}			
\newcommand{\fb}{g}			
\newcommand{\fc}{h}			
\newcommand{\ca}{a}	
\newcommand{\cb}{\alpha}		
\newcommand{\ste}[1][]{\my@opt@sub{x}{#1}}		
\newcommand{\str}[1][]{\my@opt@sub{\zeta}{#1}}	
\newcommand{\st}[1][]{\my@opt@exp{x}{#1}}		
\newcommand{\cst}{\eta}							
\newcommand{\stf}{X}							
\newcommand{\stavg}[1][]{\my@opt@exp{y}{#1}}	
\newcommand{\sttrc}[1][]{\my@opt@exp{z}{#1}}	
\newcommand{\stapp}[1][]{\my@opt@exp{\xi}{#1}}	
\newcommand{\slt}{t}
\newcommand{\fst}{\sigma}
\newcommand{\psid}[1][]{W}									
\newcommand{\psidres}[1][]{\my@opt@sub{K}{#1}}
\newcommand{\go}[1][]{\my@opt@sub{l}{#1}} 
\newcommand{\gc}[1][]{\my@opt@sub{k}{#1}} 
\newcommand{\mes}[1][]{\my@opt@exp{y}{#1}}	
\makeatother

\newcommand{\cmd}{u}

\newcommand{\noise}{\nu}
\newcommand{\nmes}{y}		
\newcommand{\vmes}{y_v} 	
\newcommand{\cout}{u}		
\newcommand{\zero}[1]{{#1}_0}

\newcommand{\finj}{s}			
\newcommand{\Finj}{S}			
\newcommand{\FFinj}{\mathcal{S}}
\newcommand{\hf}[1]{\widetilde{#1}}
\newcommand{\lf}[1]{\overline{#1}}

\newcommand{\FSderive}[1]{\frac{d#1}{d\fst}}
\newcommand{\slderive}[1]{\dot{#1}}

\newcommand{\pderive}[2]{\frac{\partial#1}{\partial#2}}

\DeclareMathOperator{\Landau}{\mathcal{O_\infty}}
\newcommand{\Exp}[1]{\exp(#1)}

\DeclareMathOperator{\sinc}{sinc}
\newcommand{\Sinc}[1]{\sinc(#1)}

\DeclareMathOperator{\psd}{PSD}
\newcommand{\Psd}[1]{\psd[#1]}

\newcommand{\avg}[2]{
	\ifx!#2\relax
		\frac{1}{#1}\int^{#1}_{0}
	\else
		\ifx/#2\relax
			\frac{1}{#1}\int^{\frac{#1}{2}}_{-\frac{#1}{2}}
		\else
			\frac{1}{#1}\int^{#2}_{#2 - #1}
		\fi
	\fi
}
\newcommand{\Avg}[2]{
	\ifx!#2\relax
		\frac{1}{#1}\int\limits^{#1}_{0}
	\else
		\ifx/#2\relax
			\frac{1}{#1}\int\limits^{\frac{#1}{2}}_{-\frac{#1}{2}}
		\else
			\frac{1}{#1}\int\limits^{#2}_{#2 - #1}
		\fi
	\fi
}
\newcommand{\res}[2]{R}

\newcommand{\rf}[1]{#1^{ref}}
\newcommand{\equ}[1]{#1^{eq}}
\newcommand{\obs}[1]{\widehat{#1}}
\newcommand{\est}[1]{\obs{#1}}

\title{\LARGE \bf Adding virtual measurements by signal injection}
\author{Pascal Combes\textsuperscript{1,2}, Al Kassem Jebai\textsuperscript{2}, François Malrait\textsuperscript{2}, Philippe Martin\textsuperscript{1} and Pierre Rouchon\textsuperscript{1}
\thanks{P. Martin is partially supported by the French Agence Nationale de la Recherche through the ANR ASTRID SCAR project ``Sensory Control of
Aerial Robots'' (ANR-12-ASTR-0033)}
\thanks{\textsuperscript{1}~P.~Combes, P.~Martin and P.~Rouchon are with the Centre Automatique et Systèmes, MINES ParisTech, PSL Research University, Paris, France
{\tt\footnotesize\{philippe.martin,pierre.rouchon\}@mines-paristech.fr}}%
\thanks{\textsuperscript{2}~P.~Combes, A.-K.~Jebai and F.~Malrait are with Schneider Toshiba Inverter Europe, 27120~Pacy-sur-Eure,~France
{\tt\footnotesize \{pascal.combes,al-kas sem.jebai,francois.malrait\}@schneider-electric.com}}
}

\begin{document}

\maketitle
\thispagestyle{empty}
\pagestyle{empty}

\begin{abstract}
We propose a method to ``create'' a new measurement output by exciting the system with a high-frequency oscillation. This new ``virtual'' measurement may be useful to facilitate the design of a suitable control law. The approach is especially interesting when the observability from the actual output degenerates at a steady-state regime of interest. The proposed method is based on second-order averaging and is illustrated by simulations on a simple third-order system. 
\end{abstract}

\section{Introduction}
The difficulty in designing a control law is often largely due to the properties of the (measured) output. For instance, controlling a system around a point where observability degenerates may be really challenging, even though the dynamics itself may remain simple. In this paper, we introduce and formalize a method able to produce an extra ``virtual'' output by exciting the system with a high-frequency oscillation. With the help of this new output, the design of the control law is likely to be much easier.

More precisely, consider the Single-Input Single-Output system
		\barCl\label{eqn:system}
			\slderive{\st} &=& \fa(\st) + \fb(\st)\cmd \ase\label{eq:state} \\
			\mes &=& \fc(\st) \ase\label{eq:output},
		\earCl
where $(\st,\cmd,\mes)$ belong to some compact subset of $\Reals[n]\times\Reals\times\Reals$; $\fa$, $\fb$ and $\fc$ are smooth enough maps (e.g. at least $\mathcal C^2$). 
We show that by superimposing to the control~$\cmd$ a fast-varying periodic signal with (small) period~$\eps$, we may consider that the so-called virtual output 
		\be
			\vmes := L_gh(\st)=\fc'(\st) \fb(\st)\label{eq:virtual}
		\ee
becomes available, and can therefore be used for designing a control law; on the other hand, the overall effect of the excitation on the original system~\eqref{eqn:system} remains small and can be ignored. Notice this a purely nonlinear phenomenon: nothing is gained when both $\fb$ and $\fc$ are linear with respect to~$\st$. Notice also the method is not restricted to affine SISO systems; it can be generalized to general MIMO systems along the same lines, at the cost of only added technicalities.

The idea of recovering some extra information through the injection of a high-frequency signal is not completely new. It has even become a standard method for the ``sensorless'' control of electrical motors at low velocity, since its introduction by~\cite{JanseL1995ITIA,CorleL1998ITIA}~(``sensorless'' here means that only the currents are measured, but neither the rotor position nor its velocity). Nevertheless the essence of the method is buried under the many technicalities and peculiarities of control of electrical motors, and its analysis is usually rather heuristic. A more rigorous but still electrical-motor-oriented analysis is proposed in~\cite{JebaiMMR2012ICDC,JebaiMMR2015IJoC}, based on second-order averaging; the present paper is a conceptualization and generalization of those ideas.

The paper is organized as follows: section~\ref{sec:averaging} lays the basis for the method; it analyzes by second-order averaging how the virtual output is created.  Section~\ref{sec:demod} shows that the virtual output can effectively be recovered by a kind of heterodyning procedure, and how it is debased by measurement noise. Finally, section~\ref{sec:simu} illustrates the interest of the method on a simple academic example.
\section{Signal injection and second-order averaging}\label{sec:averaging}
Consider the system~\eqref{eq:state} with the two outputs \eqref{eq:output}-\eqref{eq:virtual}. Assume we have designed a suitable control law
		\barCl\label{eqn:controller}
			\cout &=& \cb(\cst, \mes, \vmes, \slt) \ase \\
			\slderive{\cst} &=& \ca(\cst, \mes, \vmes, \slt) \ase,
		\earCl
with $\cst\in\Reals[p]$. By ``suitable'', we mean the closed-loop system
		\barCl\label{eqn:system:closed-loop}
		\slderive{\lf{\st}} &=& \fa(\lf{\st}) + \fb\bigl(\lf{\st})\cb(\lf{\cst}, h(\lf{\st}), L_gh(\lf{\st}), \slt\bigr) \ase \\
		\slderive{\lf{\cst}} &=& \ca\bigl(\lf{\cst}, h(\lf{\st}), L_gh(\lf{\st}), \slt\bigr) \ase
		\earCl
has the desired exponentially stable equilibrium point (or family of equilibrium points); we have changed the notation of the state to~$(\lf\st,\lf\cst)$, so as to distinguish between the solutions of~\eqref{eqn:system:closed-loop} and those of~\eqref{eqn:system:closed-loop:hf} below.

Consider now the modified control law
		\barCl\label{eqn:controllerModified}
		\cout &=&  \cb(\cst, \lf{\mes}, \lf{\vmes}, \slt)+ \finj\Bigl(\frac{\slt}{\eps}\Bigr) \ase \\
		\slderive{\cst} &=& \ca(\cst, \lf{\mes}, \lf{\vmes}, \slt) \ase\\
		\lf{\mes} &=& h(\st)+\eps\kappa\Bigl(\st,\frac{\slt}{\eps}\Bigr)+\Landau(\eps^2) \ase\label{eqn:controllerModifiedy}\\
		\lf{\vmes} &=& L_gh(\st)+\eps\kappa_v\Bigl(\st,\frac{\slt}{\eps}\Bigr)+\Landau(\eps^2) \ase\label{eqn:controllerModifiedyv}
		\earCl
where $\finj$ is a periodic function with period~$1$ and zero mean, i.e. $\int_0^1s(\fst)d\fst=0$; $\kappa$ and $\kappa_v$ are periodic with period~$1$ with respect to their second arguments and zero mean, i.e
\beno
\int_0^1\kappa(\st,\fst)d\fst=\int_0^1\kappa_v(\st,\fst)d\fst=0;
\eeno
$\eps$ is a ``small'' parameter, so that the signal superimposed to the base control law is fast-varying; $\Landau$ is the ``uniform big O'' symbol of analysis, i.e. we write $k(z,\eps)=\Landau(\eps)$ to mean $\norm{k(z,\eps)}\leq C\eps$ for some constant $C$ independent of $z$ and~$\eps$. The closed-loop system then reads
		\barCl\label{eqn:system:closed-loop:hf}
		\slderive{\st} &=& \fa(\st) + \fb(\st)\cb(\cst, \lf{\mes}, \lf{\vmes}, \slt) + \fb(\st)\finj\Bigl(\frac{\slt}{\eps}\Bigr)\ans \ase \\
		\slderive{\cst} &=& \ca(\cst, \lf{\mes}, \lf{\vmes}, \slt) \ase \\
		\lf{\mes} &=& h(\st)+\eps\kappa\Bigl(\st,\frac{\slt}{\eps}\Bigr)+\Landau(\eps^2) \ase\\
		\lf{\vmes} &=& L_gh(\st)+\eps\kappa_v\Bigl(\st,\frac{\slt}{\eps}\Bigr)+\Landau(\eps^2) .\ase
		\earCl
The goal of injecting a fast-varying oscillation is to ``create'' the virtual output~$\vmes$. We will see in corollary~\ref{cor:principle} how to choose $\kappa$ and $\kappa_v$ so that $\lf{\mes}$ and $\lf{\vmes}$ corresponds to actually available signals.
		
The following theorem characterizes the effect of signal injection by comparing the solutions of \eqref{eqn:system:closed-loop} and~\eqref{eqn:system:closed-loop:hf}.
		\begin{thm}
			\label{thm:main}
			Let $\bigl(\lf{\st}(\slt), \lf{\cst}(\slt)\bigr)$ and $\bigl(\st(\slt), \cst(\slt)\bigr)$ be the solution of the closed-loop systems \eqref{eqn:system:closed-loop} and~\eqref{eqn:system:closed-loop:hf} respectively, starting from the same initial condition. Then for all $\slt \ge 0$,
			\barCl
				\st(\slt) &=& \lf{\st}(\slt) + \eps \fb\bigl(\lf{\st}(\slt)\bigr)  \Finj\Bigl(\frac{\slt}{\eps}\Bigr) + \Landau(\eps^2)\ase \label{eq:thmx}\\
				\cst(\slt) &=& \lf{\cst}(\slt) + \Landau(\eps^2)\ase \label{eq:thmeta}\\
				\mes(\slt) &=& h\bigl(\lf{\st}(\slt)\bigr) + \eps L_gh\bigl(\lf{\st}(\slt)\bigr) \Finj\Bigl(\frac{\slt}{\eps}\Bigr) + \Landau\left(\eps[2]\right),\ase\label{eq:thmy}
			\earCl	
where $S$ is the (of course also $1$-periodic) primitive of~$s$ with zero mean, i.e.
			\barClno
			S(\fst) &:=& \int_0^\fst s(\tau)d\tau-\int_0^1\int_0^\mu s(\tau)d\tau d\mu.
			\earClno		
		\end{thm}
In other words, signal injection
\begin{itemize}
\item has a small effect of order~$\eps$ on the state variables directly affected by the input
\item has a very small effect of order~$\eps^2$ on the state variables not directly affected by the input. In many systems of interest, the input affects directly only one state variable whereas the control objective is a combination of the other stable variables; the control objective is thus hardly affected by the high-frequency excitation. This is the case for instance in electrical motors: only the fluxes (or currents) are directly affected by the inputs, while the goal is to control the rotor velocity (or position)
\item creates a small ``ripple'' of order~$\eps$ in the measured output. The amplitude of this ripple is precisely the virtual output. A procedure to extract $h\bigl(\lf{\st}(\slt)\bigr)$ and~$L_gh\bigl(\lf{\st}(\slt)\bigr)$ from~$\mes(\slt)$ is presented in section~\ref{sec:demod}; using these signals in the control law~\eqref{eqn:controllerModified} amounts to a particular choice of $\kappa$ and~$\kappa_v$, see corollary~\ref{cor:principle}.
\end{itemize}
\begin{proof}
The proof is a direct application of second-order averaging for differential equations~\cite[section 2.9]{Sanders2005}, with slow time dependence \cite[section 3.3]{Sanders2005}. We first recall the main result of this theory, and then apply it to our case.

Consider the perturbed system
\beno
\FSderive{\stf} = \eps F_1(\stf,\eps\fst,\fst)+\eps^2F_2(\stf,\eps\fst,\fst)+\Landau(\eps^3)
\eeno
with initial condition $X(0):=\lf\stf_0+\eps W(\lf\stf_0,0,0)$, where $F_1$ and $F_2$ are $T$-periodic with respect to their third arguments, and the averaged system
\be\label{eq:averagedSys}
\FSderive{\lf\stf} = \eps G_1(\lf\stf,\eps\fst)+\eps^2G_2(\lf\stf,\eps\fst)
\ee
with initial condition $\lf\stf(0):=\lf\stf_0$; finally, $G_1$, $W$ and~$G_2$ are defined by
			\barClno
			G_1(\lf\stf,\eps\fst) &:=& \frac{1}{T}\int_0^TF_1(\lf{\stf},\eps\fst,\tau)d\tau\\
			w(\lf\stf,\eps\fst,\fst) &:=& \int_0^\fst\big[F_1(\lf{\stf},\eps\fst,\tau)-G_1(\lf{\stf},\eps\fst)\bigr]d\tau\\
			W(\lf\stf,\eps\fst,\fst) &:=& w(\lf\stf,\eps\fst,\fst) - \frac{1}{T}\int_0^Tw(\lf\stf,\eps\fst,\tau)d\tau\\			
			K_2(\lf\stf,\eps\fst,\fst) &:=& F_2(\lf\stf,\eps\fst,\fst)\\
			&&+\, \partial_1F_1(\lf\stf,\eps\fst,\fst)W(\lf\stf,\eps\fst,\fst)\\
			&&-\, \partial_1W(\lf\stf,\eps\fst,\fst)G_1(\lf\stf,\eps\fst,\fst)\\
			G_2(\lf\stf,\eps\fst) &:=& \frac{1}{T}\int_0^TK_2(\lf\stf,\eps\fst,\tau)d\tau,
			\earClno
where $\partial_k$ denotes the partial derivative with respect to the $k^{th}$ argument.
The theory of second-order averaging then asserts that the solution $\stf(\fst)$ of the perturbed system and the solution $\lf\stf(\fst)$ of the perturbed system are related by
			\barClno
			\stf(\fst) &=& \lf\stf(\fst) + \eps W(\lf\stf,\eps\fst,\fst) +\Landau(\eps[2])
			\earClno
on the timescale~$\frac{1}{\eps}$. If moreover the averaged system has an exponentially stable equilibrium point with region of attraction~$D$, and if the initial condition~$\lf\stf_0$ belongs to a compact subset of~$D$, then this estimation can be continued to infinity by lemma~\ref{thm:lem1} in the appendix.

To apply this result to our case, we define $\stf := (\st, \cst)$ and rewrite~\eqref{eqn:system:closed-loop:hf} in the fast time $\fst := \frac{\slt}{\eps}$. This yields
			\beno
			\FSderive{\stf} = \eps\bigl(\lf F_1(\stf, \eps\fst) + \hf F_1(\stf)\finj(\fst)\bigr)
			+\eps^2F_2(\stf, \eps\fst,\fst) +\Landau(\eps^3),
			\eeno
where
			\barClno
			\lf F_1(\stf, \eps\fst) &:=& \begin{pmatrix}
				\fa(\st) + \fb(\st)\cb(\cdot)\\
				\ca(\cdot)
			\end{pmatrix} \\
			\hf F_1(\stf) &:=& \begin{pmatrix}
				\fb(\st) \\
				0
			\end{pmatrix} \\
			F_2(\stf, \eps\fst,\fst) &:=& \begin{pmatrix}
			\fb(\st)\bigl(\partial_2\cb(\cdot)\kappa(\st,\fst)+\partial_3\cb(\cdot)\kappa_v(\st,\fst)\bigr)\\
			\partial_2\ca(\cdot)\kappa(\st,\fst)+\partial_3\ca(\cdot)\kappa_v(\st,\fst)
			\end{pmatrix};
			\earClno
in the expressions above, $(\cdot)$ denotes $\bigl(\cst,h(\st),L_gh(\st),\eps\fst\bigr)$. We then find
			\barClno
			G_1(\lf\stf,\eps\fst) &=& \lf F_1(\lf\stf, \eps\fst)\\
			w(\lf\stf,\eps\fst,\fst) &=& \hf F_1(\lf\stf)\int_0^\fst s(\tau)d\tau\\
			W(\lf\stf,\eps\fst,\fst) &=& \hf F_1(\lf\stf)S(\fst)\\			
			K_2(\lf\stf,\eps\fst,\fst) &=& F_2(\lf\stf,\eps\fst,\fst) + \partial_1\lf F_1(\lf\stf,\eps\fst)\hf F_1(\lf\stf)S(\fst)\\
			&&-\, \partial_1\hf F_1(\lf\stf)\lf F_1(\lf\stf,\eps\fst)S(\fst)\\
			&&+\, \frac{1}{2}\partial_1\hf F_1(\lf\stf)\hf F_1(\lf\stf)\FSderive{S^2}(\fst)\\
			G_2(\lf\stf,\eps\fst) &=& 0,
			\earClno
remembering that $s$, $S$, $\kappa$ and $\kappa_v$ have zero mean. We then rewrite the averaged system~\eqref{eq:averagedSys} in the slow time~$\slt$ to find
\beno
\dot{\lf\stf} = G_1(\lf\stf,\slt)+\eps G_2(\lf\stf,\slt) = \lf F_1(\lf\stf,\slt),
\eeno
which is exactly the closed-loop system~\eqref{eqn:system:closed-loop}. Moreover, $\stf$ and $\lf\stf$ are related by
			\beno
				\stf(\slt) = \lf{\stf}(\slt) + \eps\hf F_1\bigl(\lf{\stf}(\slt)\bigr) \Finj\Bigl(\frac{\slt}{\eps}\Bigr) + \Landau(\eps[2]),	
			\eeno
which is~\eqref{eq:thmx}-\eqref{eq:thmeta}. Finally, we get~\eqref{eq:thmy} by injecting~\eqref{eq:thmx} in the expression of the output,
			\barClno
				\mes(t) &=& \fc\bigl(\st(t)\bigr)\\
				&=& \fc\biggl(\lf{\st}(t) + \eps\fb\bigl(\lf\st(t)\bigr)\Finj\Bigl(\frac{\slt}{\eps}\Bigr) + \Landau(\eps[2])\biggr) \\
				&=& \fc\bigl(\lf\st(t)\bigr) + \eps L_gh\bigl(\lf\st(t)\bigr)\Finj\Bigl(\frac{\slt}{\eps}\Bigr) + \Landau(\eps[2]).
			\earClno
We have assumed without loss of generality that $S(0)=0$, which implies $\stf(0)=\lf\stf(0)$; this is always possible by suitably shifting in time the signal~$s$.
		\end{proof}

\begin{crl}\label{cor:principle}
Assume the signals $h\bigl(\lf{\st}(\slt)\bigr)$ and $L_gh\bigl(\lf{\st}(\slt)\bigr)$ in~\eqref{eq:thmy} are available. Then the contol law~\eqref{eqn:controllerModified} is actually implementable by choosing
			\barClno
				\kappa\Bigl(\st,\frac{\slt}{\eps}\Bigr) &:=& - L_gh(\st) \Finj\Bigl(\frac{\slt}{\eps}\Bigr)\\
				\kappa_v\Bigl(\st,\frac{\slt}{\eps}\Bigr) &:=& - L^2_gh(\st) \Finj\Bigl(\frac{\slt}{\eps}\Bigr).
			\earClno

\end{crl}
		\begin{proof} Using~\eqref{eq:thmx} we obviously have
					\barClno
					h(\st)-\eps L_gh(\st) \Finj\Bigl(\frac{\slt}{\eps}\Bigr) &=& 
					h\biggl(\lf{\st} + \eps\fb\bigl(\lf\st\bigr)\Finj\Bigl(\frac{\slt}{\eps}\Bigr) + \Landau(\eps[2])\biggr)\\
					&&-\,\eps L_gh\bigl(\lf{\st} + \Landau(\eps)\bigr)\Finj\Bigl(\frac{\slt}{\eps}\Bigr)\\
					&=& h(\lf\st) + \Landau(\eps^2)\\
					L_gh(\st)-\eps L^2_gh(\st) \Finj\Bigl(\frac{\slt}{\eps}\Bigr) &=& L_gh(\lf\st) + \Landau(\eps^2),
					\earClno
i.e. $\lf{\mes}=h(\lf\st)$ in \eqref{eqn:controllerModifiedy} and $\lf{\vmes}=L_gh(\lf\st)$ in \eqref{eqn:controllerModifiedyv}.
		\end{proof}
\section{Extracting the outputs}\label{sec:demod}
We now turn to extracting the information contained in~\eqref{eq:thmy}. In other words, given a signal of the form
	\beno\label{eqn:measurement}
		\nmes(\slt) = \lf{\mes}(\slt) + \eps\overline{y_v}(\slt)\Finj\Bigl(\frac{\slt}{\eps}\Bigr) + \noise(\slt), 
	\eeno
and corrupted by the measurement noise~$\nu$, we would like to recover its components $\lf{\mes}(\slt)$ and $\overline{y_v}(\slt)$. We will show this can be achieved by the estimators
	\barCl
		\est{\lf{\mes}}(\slt) &:=& \avg{n\eps}{\slt} \mes(\tau) d\tau \ase \label{eqn:demod:ybar} \\
		\est{\overline{y_v}}(\slt) &:=& \frac{1}{\eps} \frac{\avg{n\eps}{\slt} \Bigl(\mes\bigl(\tau - \frac{n\eps}{2}\bigr) - \est{\lf{\mes}}(\tau)\Bigr)\Finj\Bigl(\frac{\tau - \frac{n\eps}{2}}{\eps}\Bigr) d\tau}{\avg{n\eps}{\slt} \Finj^2(\frac{\tau}{\eps}) d\tau}, \ase\ans \label{eqn:demod:ytilde}
	\earCl
with $n\in\mathbb{N}$. We study the accuracy of these estimators without noise in~\ref{sec:demod:extract}, and their sensitivity to noise in~\ref{sec:demod:noise}. Indeed, since the noise is additive and enters the estimators linearly, the two issues can be studied independently.

	\subsection{Accuracy of the estimators}
	\label{sec:demod:extract}
		\begin{prop}
			The accuracy of the estimators \eqref{eqn:demod:ybar} and \eqref{eqn:demod:ytilde} is as follows
			\barCl
				\est{\lf{\mes}}(\slt) &=& \lf{\mes}\Bigl(\slt - \frac{n\eps}{2}\Bigr) + \Landau\bigl(n^2\eps^2\bigr) = \lf{\mes}(\slt) + \Landau(n\eps) \ans \ase \\
				\est{\overline{y_v}}(\slt) &=& \overline{y_v}(\slt) + \Landau(n^2\eps) \ase
			\earCl
		\end{prop}
		\begin{proof}
			The signals $\lf{\mes}$ and $\overline{y_v}$ are assumed sufficiently smooth and well-behaved; more precisely, we assume they can be written as Taylor series with integral remainder
			\barClno
				\lf{\mes}(\slt - \sigma) &=& \lf{\mes}(\slt) - \sigma\slderive{\lf{\mes}}(\slt) + \sigma^2\lf{\res{\y}{2}}(\slt, \sigma) \\
				\overline{y_v}(\slt - \sigma) &=& \overline{y_v}(\slt) - \sigma\slderive{\overline{y_v}}(\slt) + \sigma^2\hf{\res{\y}{2}}(\slt, \sigma),
			\earClno
			with $\lf{\res{\y}{2}}(\slt, \sigma)$ and $\hf{\res{\y}{2}}(\slt, \sigma)$ uniformly (with respect to~$t$) bounded for $\sigma \in [0, n\eps]$. Using these expressions, we find
			\barClno
				\est{\lf{\mes}}(t) 
				&=& \Avg{n\eps}{!} \lf{\mes}(\slt - \sigma) d\sigma + \eps\Avg{n\eps}{!} \overline{y_v}(\slt - \sigma)\Finj\Bigl(\frac{\slt - \sigma}{\eps}\Bigr) d\sigma \\
				&=& \lf{\mes}(\slt)\avg{n\eps}{!} d\sigma - \slderive{\lf{\mes}}(\slt)\avg{n\eps}{!} \sigma d\sigma \\
				& & + \Avg{n\eps}{!} \sigma^2\lf{\res{\y}{2}}(\slt, \sigma) d\sigma + \eps\overline{y_v}(\slt)\Avg{n\eps}{!} \Finj\Bigl(\frac{\slt - \sigma}{\eps}\Bigr) d\sigma \\
				& & - \eps\slderive{\overline{y_v}}(\slt)\avg{n\eps}{!} \sigma\Finj\Bigl(\frac{\slt - \sigma}{\eps}\Bigr) d\sigma \\
				& &+ \eps\avg{n\eps}{!} \sigma^2\hf{\res{\y}{2}}(\slt, \sigma)\Finj\Bigl(\frac{\slt - \sigma}{\eps}\Bigr) d\sigma \\
				&=& \lf{\mes}(\slt) - \frac{n\eps}{2}\slderive{\lf{\mes}}(\slt) + \frac{(n\eps)^2}{3}\Landau(1) \\
				& & + \eps^2 \FFinj(\slt) - \frac{n^2\eps^3}{3}\sup\limits_{\sigma\in[0, n\eps]}\bigl(\abs{S(\slt-\sigma)}\bigr)\Landau(1) \\
				&=& \lf{\mes}\Bigl(\slt - \frac{n\eps}{2}\Bigr) + \Landau\bigl(n^2\eps^2\bigr)
			\earClno
			where $\FFinj$ is the primitive of $\Finj$ with zero mean. Consequently, we have
			\barClno
				\frac{\mes\Bigl(\slt - \frac{n\eps}{2}\Bigr) - \est{\lf{\mes}}(\slt)}{\eps} &=& \overline{y_v}\Bigl(\slt - \frac{n\eps}{2}\Bigr)\Finj\Biggl(\frac{\slt - \frac{n\eps}{2}}{\eps}\Biggr) + \Landau\bigl(n^2\eps\bigr)
			\earClno
			which leads to 
			\barClno
				\lf{\Finj^2}\est{\overline{y_v}}(\slt) &=& \avg{n\eps}{\slt} \overline{y_v}\Bigl(\tau - \frac{n\eps}{2}\Bigr)\Finj\Biggl(\frac{\tau - \frac{n\eps}{2}}{\eps}\Biggr)^2 + \Landau\bigl(n^2\eps\bigr) \\
				&=& \avg{n\eps}{!} \overline{y_v}\Bigl(\slt - \sigma - \frac{n\eps}{2}\Bigr) \Finj\Biggl(\frac{\slt - \sigma - \frac{n\eps}{2}}{\eps}\Biggr)^2 d\sigma \\
				& & + \Landau\bigl(n^2\eps\bigr) \\
				&=& \overline{y_v}\Bigl(\slt - \frac{n\eps}{2}\Bigr) \avg{n\eps}{!} \Finj\Biggl(\frac{\slt - \sigma - \frac{n\eps}{2}}{\eps}\Biggr)^2 d\sigma \\
				& & - \slderive{\overline{y_v}}\Bigl(\slt - \frac{n\eps}{2}\Bigr) \avg{n\eps}{!} \sigma \Finj\Biggl(\frac{\slt - \sigma - \frac{n\eps}{2}}{\eps}\Biggr)^2 d\sigma \\
				& & + \Avg{n\eps}{!} \sigma^2 \hf{\res{\mes}{2}}\Bigl(\slt - \frac{n\eps}{2}, \sigma\Bigr) \Finj\Biggl(\frac{\slt - \sigma - \frac{n\eps}{2}}{\eps}\Biggr)^2 d\sigma \\
				& & + \Landau\bigl(n^2\eps\bigr) \\
				&=& \lf{\Finj^2}\overline{y_v}\Bigl(\slt - \frac{n\eps}{2}\Bigr) - \eps\slderive{\overline{y_v}}\Bigl(\slt - \frac{n\eps}{2}\Bigr)\FFinj_2\Bigl(\slt - \frac{n\eps}{2}\Bigr) \\
				& & + \frac{(n\eps)^2}{3}\sup\limits_{\sigma\in[0, n\eps]}\Bigl(\abs{S^2\Bigl(\slt - \sigma - \frac{n\eps}{2}\Bigr)}\Bigr)\Landau(1) \\
				& & + \Landau\bigl(n^2\eps\bigr)  \\
				&=& \lf{\Finj^2}\overline{y_v}(\slt) + \Landau(n^2\eps)
			\earClno
			where $\FFinj_2$ is the primitive of $\Finj^2$ with zero mean.
		\end{proof}
		
		\begin{rmk}
			The simpler formula
			\be\label{eqn:demod:ytilde:bad}
				\est{\overline{y_v}}(\slt) = \frac{1}{\eps} \frac{\avg{n\eps}{\slt} \mes(\tau)\Finj(\frac{\tau}{\eps}) d\tau}{\avg{n\eps}{\slt} \Finj(\frac{\tau}{\eps})^2 d\tau},\ase 
			\ee
			proposed in~\cite{JebaiMMR2012ICDC,JebaiMMR2015IJoC} by considering $\lf{\mes}$ and $\overline{y_v}$ are constant on one period $\eps$ of the high-frequency signal is less precise than~\eqref{eqn:demod:ytilde}, since valid only if $\lf{\mes}$ and $\overline{y_v}$ vary very slowly.
		\end{rmk}
		
	\subsection{Sensitivity to noise}
	\label{sec:demod:noise}
		As the virtual measurement estimate is scaled by a factor $\eps$ it may be more sensitive to noise than the original measurement. To study this issue, we assume the measurement noise $\nu$ is white with Power Spectral Density~$\Psd{\noise}$.  For simplicity, we moreover consider~\eqref{eqn:demod:ytilde:bad} instead of~\eqref{eqn:demod:ytilde}; it is nevertheless possible to conduct a similar analysis for~\eqref{eqn:demod:ytilde} at the cost of added technicalities. The additive noise~$\nu$ obviously creates additive noises on the estimates $\est{\lf{\mes}}$ and $\est{\overline{y_v}}$, denoted  respectively $\lf{\noise}$ and $\hf{\noise}$. Their PSDs are given by
		\barClno
			\Psd{\lf{\noise}}(\omega) &=& \Psd{\noise}(\omega) \abs{H(\jmath \omega)}^2 \\
			\Psd{\hf{\noise}}(\omega) &=& \frac{1}{\eps[2]\lf{\Finj^2}^2} \Psd{\Finj\noise}(\omega) \abs{H(\jmath \omega)}^2,
		\earClno
		where $H(\jmath \omega) := \frac{1 - e^{-\jmath n\eps\omega}}{\jmath n\eps\omega} = \Exp{-\jmath\frac{n\eps}{2}\omega}\Sinc{\frac{n \eps}{2}\omega}$ is the transfer function of the sliding average. It remains to compute $\Psd{\Finj \noise}$, i.e. the Fourier transform of the autocorrelation of $\Finj\bigl(\frac{\slt}{\eps}\bigr) \noise(\slt)$ which is non-stationary. The autocorrelation is
		\barClno
			R(\tau) &=& \lim\limits_{\Delta T \rightarrow \infty} \frac{1}{2 \Delta T} \int\limits_{-\Delta T}^{\Delta T} \Finj\Bigl(\frac{\slt}{\eps}\Bigr) \Finj\Bigl(\frac{t + \tau}{\eps}\Bigr) \noise(t)\noise(t + \tau) dt \\
				&=& \lim\limits_{\Delta T \rightarrow \infty} \frac{1}{2 \Delta T} \int_{-\Delta T}^{\Delta T} \Finj\Bigl(\frac{\slt}{\eps}\Bigr) \Finj\Bigl(\frac{t + \tau}{\eps}\Bigr) dt \\
				& & \lim\limits_{\Delta T \rightarrow \infty} \frac{1}{2 \Delta T} \int_{-\Delta T}^{\Delta T}\noise(t)\noise(t + \tau) dt \\
				&=& \Biggl(\avg{\eps}{!} \Finj\Bigl(\frac{\slt}{\eps}\Bigr) \Finj\Bigl(\frac{t + \tau}{\eps}\Bigr) dt\Biggr)\Psd{\noise} \delta(\tau) \\
				&=& \lf{\Finj^2}\Psd{\noise} \delta(\tau),
		\earClno
		since $\Finj$ and $\noise$ are independent. The signal $\noise(t)\Finj\bigl(\frac{\slt}{\eps}\bigr)$ thus behaves in average as a white noise with a reduced PSD. 
		As the cardinal sine function is bounded by the inverse function after one period, the gain of the sliding average over a time range $n\eps$ is bounded by the gain of a low-pass filter with bandwidth $2 \frac{2\pi}{n\eps}$. The estimators $\est{\lf{\mes}}$ and $\est{\overline{y_v}}$ thus have a built-in filtering effect. To decrease the influence of measurement noise, we can therefore
		\begin{itemize}
			\item increase the amplitude of the high-frequency oscillation~$s$, without exceeding~$\Landau(1)$
			\item average on a longer time by using a larger~$n$,  at the cost of a larger delay.
		\end{itemize}
		\begin{figure}
			\centering
			\setlength\figurewidth{6.5cm}
			\setlength\figureheight{2cm}
			\subcaptionbox{Due to its high frequency, the injected signal looks like a thick solid line.\label{fig:cmd:orig}}{\includegraphics{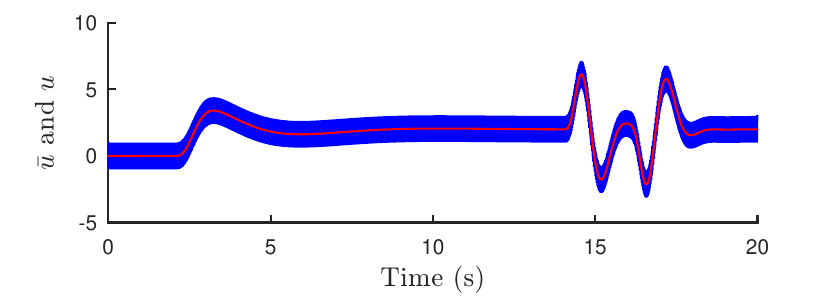}}
			\subcaptionbox{Zoom on fig.~\ref{fig:cmd:orig}, showing the injected signal is indeed a square wave. \label{fig:cmd:zoom}}{\includegraphics{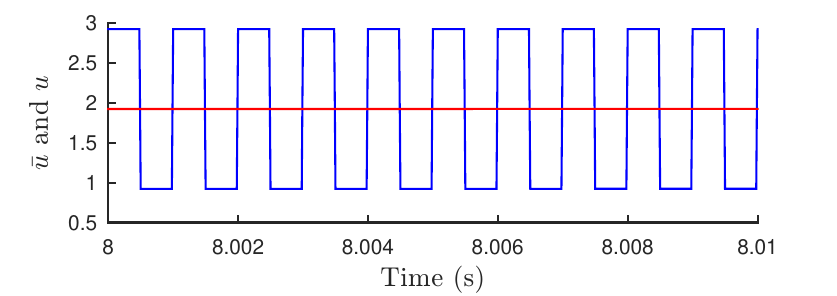}}
			\caption{The control with (blue solid line) and without (red solid line) high-frequency injection}
			\label{fig:cmd}
		\end{figure}
		\begin{figure}
			\centering
			\setlength\figurewidth{6.5cm}
			\setlength\figureheight{2cm}
			\subcaptionbox{The state variable $\ste[1]$ (green and blue lines), its estimate $\est{\overline{y_v}}$ (dashed magenta line) and its reference $\rf{{\ste[1]}}$ (dashed red line).\label{fig:state:x1}}{\includegraphics{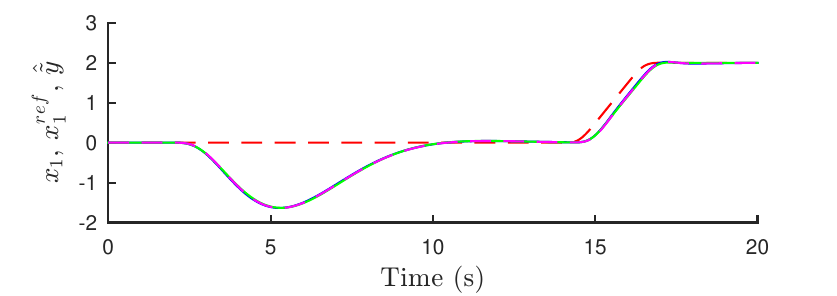}}
			\subcaptionbox{The state variable $\ste[2]$.\label{fig:state:x2}}{\includegraphics{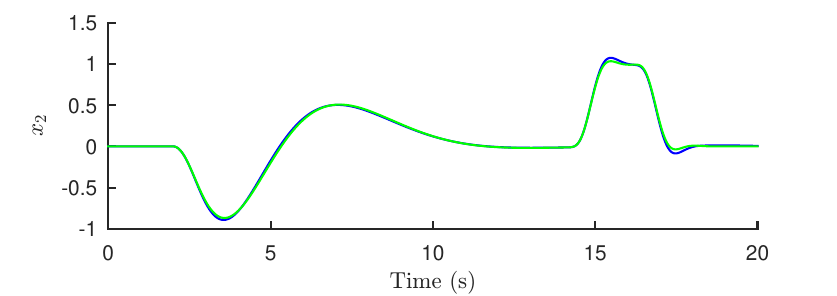}}
			\subcaptionbox{The state variable $\ste[3]$.\label{fig:state:x3}}{\includegraphics{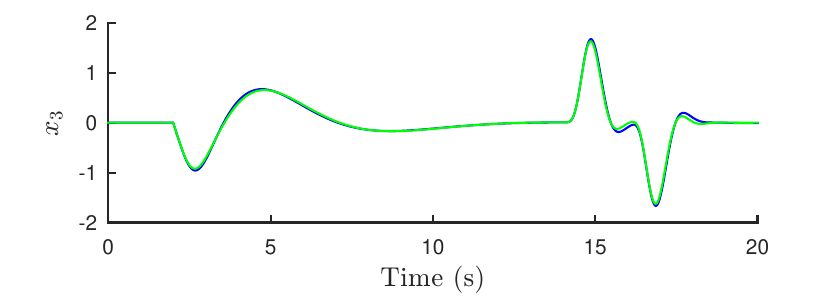}}
			\subcaptionbox{Zoom on fig.~\eqref{fig:state:x3}, showing the influence of high-frequency injection on the state variable $\ste[3]$. \label{fig:state:x3:zoom}}{\includegraphics{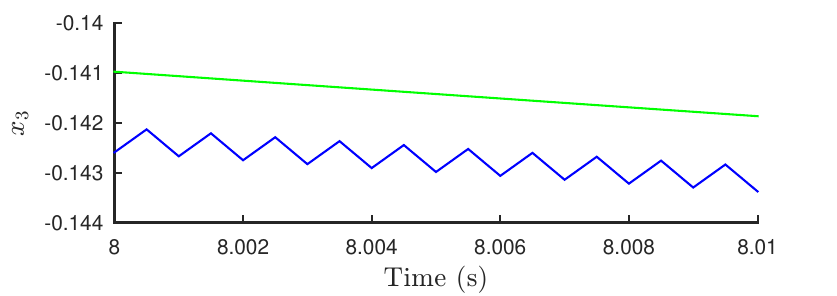}}
			\caption{The state of the system \eqref{eqn:basic:system} controlled using the real (solid green line) and the estimated (solid blue line) value of the virtual measurement.}
			\label{fig:state}
		\end{figure}
		\begin{figure}
			\centering
			\setlength\figurewidth{6.5cm}
			\setlength\figureheight{2cm}
			\subcaptionbox{The measurement $\mes$ (blue line) and its average $\est{\lf{\mes}}$ (red line) extracted using the demodulation procedure of section~\ref{sec:demod:extract}. \label{fig:mes:orig}}{\includegraphics{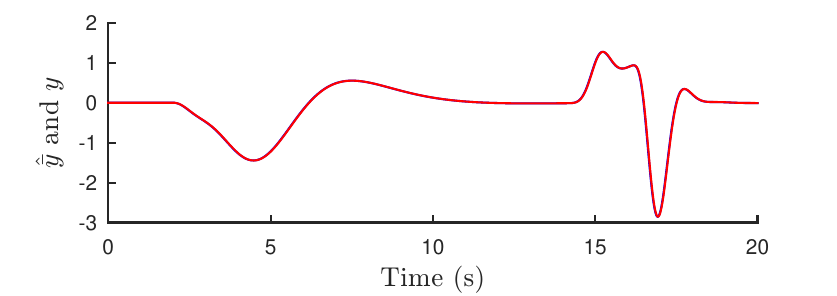}}
			\subcaptionbox{Zoom on fig.~\ref{fig:mes:orig}, showing the ripples caused by high-frequency injection on the measurement~$\mes$.\label{fig:mes:zoom}}{\includegraphics{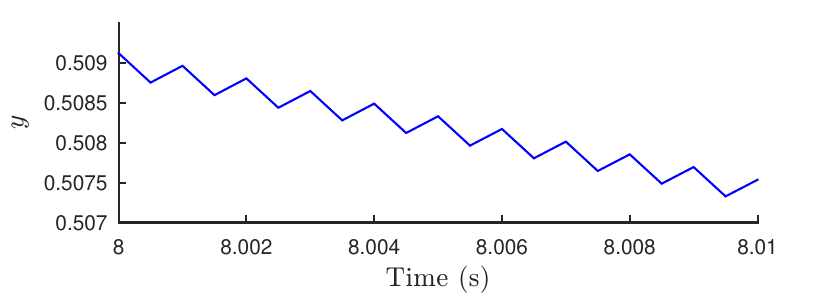}}
			\caption{Effect of injection on the measurement.}
			\label{fig:mes}
		\end{figure}			
\section{A worked example}
\label{sec:simu}
		As an example we consider the simple system
		\barCl\label{eqn:basic:system}
			\dot x_1 &=& x_2 \ase \\
			\dot x_2 &=& x_3 \ase \\
			\dot x_3 &=& u+d \ase \\ 
			y &=& x_2+x_1x_3, \ase 
		\earCl
		where $d$ is an unknown disturbance. The goal is to control the variable~$x_1$, while rejecting the disturbance, with time responses of about a few seconds. Moreover, we would like to operate around the equilibrium point $\bigl(\equ{\ste[1]}, \equ{\ste[2]}, \equ{\ste[3]}\bigr) = \bigl(\rf{\ste[1]}, 0, 0\bigr)$ determined by the reference constant~$\rf{\ste[1]}$;
		by definition, $u^{eq}+d^{eq}$ and all its derivatives are zero at this equilibrium point.
		However the system is clearly not observable at this point, since
		\begin{IEEEeqnarray*}{rCl'l}
			\pderive{\mes}{\ste} &=& \bigl(0, 1, \rf{\ste[1]}\bigr)\\
			\pderive{\dot{\mes}}{\ste} &=& (0, 0, 1)\\
			\pderive{{\mes^{(k)}}}{\ste} &=& (0, 0, 0), &k\ge2,
		\end{IEEEeqnarray*}
hence $\ste[1]$ can not be recovered from the output and its derivatives.

The control objective can nevertheless be met rather easily with the method proposed in this paper. Indeed, the virtual output is
	\begin{IEEEeqnarray*}{C}
		y_v = \begin{pmatrix}x_3 & 1 & x_1\end{pmatrix}\begin{pmatrix}0\\ 0\\ 1\end{pmatrix}=x_1.
	\end{IEEEeqnarray*}
With this new measurement, the system is completely linear and can be easily controlled, without using the original measurement~$\mes$. We use a classical controller-observer design, with an estimation of the perturbation by the observer to yield an implicit integral effect. The observer reads
		\begin{IEEEeqnarray*}{rCl}
			\dot{\hat x}_1 &=& \hat x_2 + l_1(\vmes - \hat x_1)\\
			\dot{\hat x}_2 &=& \hat x_3 + l_2(\vmes - \hat x_1)\\
			\dot{\hat x}_3 &=& u+\hat d + l_3(\vmes - \hat x_1)\\
			\dot{\hat d} &=& l_d(\vmes - \hat x_1),\label{eq:exoutput}
		\end{IEEEeqnarray*}
		and the controller 
		\begin{IEEEeqnarray*}{rCl}
			u &=& -k_1\hat x_1-k_2\hat x_2-k_3\hat x_3-k_d\hat d+kx_1^{ref}.
		\end{IEEEeqnarray*}
The gains are chosen such that the eigenvalues of the observer are $(-1.31,-0.80,-0.54\pm0.63i)$
and those of the controller are $(-6.06,-3.03\pm5.25i)$, which ensures a time response in disturbance rejection of a few seconds and a reasonable robustness; the controller is faster than the observer, in accordance with dual Loop Transfer Recovery (recovery at the plant output). Setting $\eta:=(\hat x_1,\hat x_2,\hat x_3,\hat d)^T$, this controller-observer can obviously be written as
		\begin{IEEEeqnarray}{rCl}\label{eqn:excontroller}
			u &=& -K\eta+kx_1^{ref}\ase\\
			\dot\eta &=& M\eta+Nx_1^{ref}+Ly_v,\ase
		\end{IEEEeqnarray}
which is indeed a particular form of~\eqref{eqn:controller}. Following section~\ref{sec:averaging}, we now modify~\eqref{eqn:excontroller} into
		\begin{IEEEeqnarray*}{rCl}
			u &=& -K\eta+kx_1^{ref}+\finj\Bigl(\frac{\slt}{\eps}\Bigr)\\
			\dot\eta &=& M\eta+Nx_1^{ref}+L\est{\overline{y_v}},
		\end{IEEEeqnarray*}
where  $\est{\overline{y_v}}$ is obtained from the actual measurement~\eqref{eq:exoutput} thanks to the estimator~\eqref{eqn:demod:ytilde}; this modified controller then has the form~\eqref{eqn:controllerModified}. We choose for the injected signal~$\finj$ a square wave of amplitude~$1$ and frequency~$1kHz$ (i.e. $\eps:= 10^{-3}$), which ensures the oscillation is fast with respect to the time constants of the closed-loop system; $n:=10$ is used in the demodulating filter.

The test scenario is the following: at $t=0$, the system starts at rest at the origin; at $t=2$, a disturbance $d$ of magnitude~$-2$ is applied; at $t=14$ the reference $x_1^{ref}$ is set to a ramp with slope~$1$, filtered by a low-pass filter with a $1Hz$
 bandwidth. Notice that the virtual output method relies on a truly nonlinear effect, though everything appears here to be linear. It is not possible to track fast and large references, since the estimators~\eqref{eqn:demod:ybar}-\eqref{eqn:demod:ytilde} assume reasonably slowly-varying signals.
	In fig.~\ref{fig:state} we compare the performance of the controller fed with the estimate $\est{\overline{y_v}}$ of the virtual output and fed with the (unavailable)
true variable~$x_1$. It is nearly impossible to distinguish the two situations on the responses of $x_1$ and $x_2$, since the error is~$\Landau(\eps^2)$; the $\Landau(\eps)$ ripples are only visible on~$x_3$. In fig.~\ref{fig:cmd} the control is shown; the frequency of the injection signal is so high that it looks like a thick solid line, but by zooming see that it is indeed a square wave. Finally, fig.~\ref{fig:mes} shows the $\Landau(\eps)$ ripples in the measurement, from which the virtual measurement is obtained.
		To investigate the sensitivity of the method to measurement noise, we now superimpose on the measurement $\mes$ a band-limited white noise with sample time $2\!\times\!10^{-5}$ and power $2\!\times\!10^{-11}$. Fig.~\ref{fig:state:noise} shows how this noise affects the virtual measurement and the controlled output $\ste[1]$. We check the results of section \ref{sec:demod:noise} by comparing them to what would be obtained with the real value of $\ste[1]$ with an equivalent noise. As can be seen in fig.~\ref{fig:mes:noise}, the ripples are buried into the noise, yet the behavior remains satisfactory.

This example demonstrates the relevance of the virtual output approach, and its good behavior even when operating near the non-observability region.

	\begin{figure}
		\centering
		\setlength\figurewidth{6.5cm}
		\setlength\figureheight{2cm}
		\includegraphics{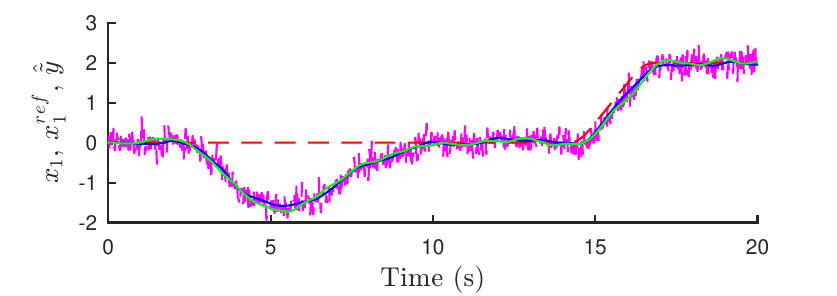}
		\caption{The state variable $\ste[1]$ controlled using the virtual measurement or the real value of $\ste[1]$; its estimate $\est{\overline{y_v}}$ (dashed magenta line), and its reference $\rf{{\ste[1]}}$ (dashed red line).}
		\label{fig:state:noise}
	\end{figure}
	
	\begin{figure}
		\centering
		\setlength\figurewidth{6.5cm}
		\setlength\figureheight{2cm}
		\subcaptionbox{The measurement $\mes$ with (green line) and without (blue line) noise, and its average $\est{\lf{\mes}}$ (red line) extracted using the demodulation procedure of~\ref{sec:demod:extract}. The effects of noise and high-frequency injection are barely visible at this scale. \label{fig:mes:noise:orig}}{\includegraphics{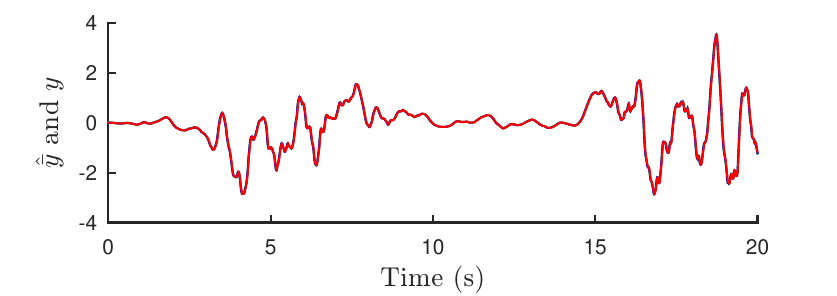}} 
		\subcaptionbox{Zoom on figure \ref{fig:mes:noise:orig}, illustrating the approach works even though the ``useful'' ripples are buried in the measurement noise. \label{fig:mes:noise:zoom}}{\includegraphics{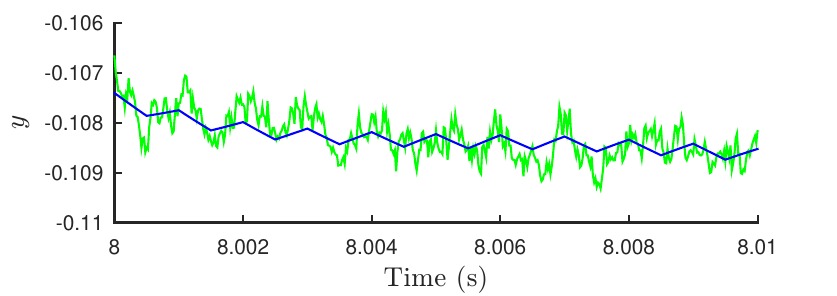}}
		\caption{$\mes$, with (green line) and without (blue line) noise, and the estimate $\est{\lf{\mes}}$ (red line).}
		\label{fig:mes:noise}
	\end{figure}
\addtolength{\textheight}{-7.35cm}   
\section*{appendix: second order averaging for exponentially stable systems}
The approximations given by first- and second-order averaging are a priori valid only  on the timescale $\frac{1}{\eps}$. However, with the additional assumption of exponential stability of the averaged system, they can be continued to infinity. This is well-known for first-order averaging, see e.g. \cite[Theorem 5.5.1]{Sanders2005}, recalled below.
		\begin{lem}
			Consider the two systems
			\barCl\label{eqn:first}
				\slderive{\st} &=& \eps F_1(\st, \slt) \ase \label{eqn:first:orig} \\
				\slderive{\sttrc} &=& \eps G_1(\sttrc) \ase \label{eqn:first:avg},
			\earCl
			where $F_1$ is $T$-periodic with respect to $\slt$ and $G_1$ is its average on one period. 
			We suppose that the origin is an exponentially stable equilibrium for \eqref{eqn:first:avg}. Then, there exists a compact neighborhood $\mathcal{V}$ of the origin, such that $\forall \zero{\sttrc} \in \mathcal{V}$ the solution $\sttrc(\slt)$ of \eqref{eqn:first:avg} with initial condition $\zero{\sttrc}$ at $\zero{\slt}$ is valid on $[\zero{\slt}, +\infty[$.
			Besides, there exists $\zero{\eps}>0$ such that $\forall \eps \in [0, \zero{\eps}[$, the solution $\st(\slt)$ of \eqref{eqn:first:orig} with initial condition $\zero{\st} = \zero{\sttrc}$ at $\zero{\slt}$ is valid on $[\zero{\slt}, +\infty[$ and $\exists C \in \mathbb{R}^+$ such that
			\beno
				\sup\limits_{\slt \in [\zero{\slt}, +\infty[} \norm{\st(\slt) - \sttrc(\slt)} < C\eps.
			\eeno
		\end{lem}
		
		The following lemma, which does not seem to exist in the literature, extends this result to the case of second-order averaging.
		\begin{lem}\label{thm:lem1}
			Consider the two systems
			\barCl\label{eqn:second}
				\slderive{\st} &=& \eps F_1(\st, \slt) + \eps^2F_2(\st, \slt) \ase \label{eqn:second:orig} \\
				\slderive{\sttrc} &=& \eps G_1(\sttrc) + \eps^2G_2(\sttrc) \ase \label{eqn:second:avg}
			\earCl
			where $F_1$ and $F_2$ are $T$-periodic with respect to $\slt$ and $G_1$ and $G_2$ are obtained as in section \ref{sec:averaging}.
			We suppose that the origin is an exponentially stable equilibrium for \eqref{eqn:first:avg}. Then, there exists a compact neighborhood $\mathcal{V}$ of the origin and $\zero{\eps}>0$ such that $\forall \eps \in [0, \zero{\eps}[$ and $\forall \zero{\sttrc} \in \mathcal{V}$, the solution $\sttrc(\slt)$ of \eqref{eqn:second:avg} with initial condition $\zero{\sttrc}$ at $\zero{\slt}$ and the solution $\st(\slt)$ of \eqref{eqn:second:orig} with initial condition $\zero{\st} = \zero{\sttrc} + \eps\psid[1](\zero{\sttrc}, 0)$ at $\zero{\slt}$ are valid on $[\zero{\slt}, +\infty[$ and $\exists C \in \mathbb{R}^+$ such that
			\beno
				\sup\limits_{\slt \in [\zero{\slt}, +\infty[} \norm{\st(\slt) - \sttrc(\slt) - \eps\psid[1](\sttrc, \slt)} < C\eps^2,
			\eeno
			where $\psid[1]$ is defined as in \ref{sec:averaging}.
		\end{lem}
		\begin{proof}
			$\mathcal{V}$ is chosen as a compact included in the domain of attraction of the origin. Then, $\exists \eps_1 > 0$ such that $\forall \eps \in [0, \eps_1[$ \eqref{eqn:second:avg} admits a solution on $[\zero{\slt}, +\infty[$. Besides, we know from \cite[Lemma~2.9.1]{Sanders2005} that \eqref{eqn:second:orig} can be transformed into 
			\beno
				\slderive{\stavg} = \eps G_1(\stavg) + \eps^2G_2(\stavg)
			\eeno
			by the change of variables $\st = \stavg + \eps\psid[1](\stavg)$, and thus, $\exists \eps_2 > 0$ such that $\forall \eps \in [0, \eps_2[$ \eqref{eqn:second:orig} admits a solution on $[\zero{\slt}, +\infty[$.
			
			We call $\stapp(\slt) = \sttrc(\slt) + \eps\psid[1](\sttrc, \slt)$. We know from \cite[Lemma~2.9.2]{Sanders2005} that $\exists L \in \mathbb{R}^+$ such that
			\beno
				\forall \slt \in \Bigl[\zero{\slt}, \zero{\slt}+\frac{L}{\eps}\Bigr[ \quad \norm{\st(\slt) - \stapp(\slt)} < C\eps^2.
			\eeno
			We partition the time into segments of length $\frac{L}{\eps}$
			\beno
				\bigcup_{m\in \mathbb{N}} I_m = \bigcup_{m\in \mathbb{N}} \Bigl[\zero{\slt} + m\frac{L}{\eps}, \zero{\slt} + (m + 1)\frac{L}{\eps}\Bigr].
			\eeno
			On each segment $I_m$ we define $\sttrc_m$ as the solution of the truncated averaged equation \eqref{eqn:second:avg} with initial condition such that $\st_m\bigl(\zero{\slt} + m\frac{T}{\eps}\bigr) = \sttrc_m\bigl(\zero{\slt} + m\frac{T}{\eps}\bigr) + \eps\psid[1]\Bigl(\sttrc_m\bigl(\zero{\slt} + m\frac{T}{\eps}\bigr), \slt\Bigr)$ and also $\stapp_m := \sttrc_m + \eps\psid[1](\sttrc_m, \slt)$.
			We note $\norm{\cdot}_{I_m} = \sup\limits_{t \in I_m} \norm{\cdot}$.
			
			$m\in \mathbb{N}$ is now fixed. From the theorem of second-order averaging \cite[Theorem 2.9.2]{Sanders2005} $\exists k$ such that
			\beno
				 \quad \norm{\stapp_m - \st}_{I_m} \leq k\eps[2]
			\eeno
			and form the definition of the truncated pseudo-identity transformation $\mathcal{W}(\sttrc, \slt) := \sttrc + \eps\psid[1](\sttrc, \slt)$, $\exists \lambda$ such that
			\barClno
				\norm{\stapp - \stapp_m}_{I_m} &=& \norm{\mathcal{W}(\sttrc, \slt) - \mathcal{W}(\sttrc_m, \slt)}_{I_m} \\
				&\leq& (1 + \eps\lambda)\norm{\sttrc - \sttrc_m}_{I_m} \\
				\norm{\sttrc - \sttrc_m}_{I_m} &=& \norm{\mathcal{W}^{-1}(\stapp, \slt) - \mathcal{W}^{-1}(\stapp_m, \slt)}_{I_m} \\
				&\leq& (1 + \eps\lambda)\norm{\stapp - \stapp_m}_{I_m}.
			\earClno
			Besides, \cite[Lemma 5.2.7]{Sanders2005} implies
			\beno
				\norm{\sttrc - \sttrc_m}_{I_{m}} \leq \kappa \norm{\sttrc\Bigl(\zero{\slt} + m\frac{L}{\eps}\Bigr) - \sttrc_m\Bigl(\zero{\slt} + m\frac{L}{\eps}\Bigr)},
			\eeno
			which can be rewritten by prolonging $\sttrc_m$ on $I_{m-1}$
			\beno
				\norm{\sttrc - \sttrc_m}_{I_{m}} \leq \kappa \norm{\sttrc - \sttrc_m}_{I_{m - 1}},
			\eeno
			where $\kappa$ can be made as small as desired by taking $\eps$ small enough. By applying the triangle inequality, we find
			\barClno
				\norm{\stapp - \st}_{I_m} &\leq& \norm{\stapp - \stapp_m}_{I_m} + \norm{\stapp_m - \st}_{I_m} \\
										  &\leq& (1 + \eps\lambda)\norm{\sttrc - \sttrc_m}_{I_m} + k\eps[2] \\
										  &\leq& (1 + \eps\lambda)\kappa\norm{\sttrc - \sttrc_m}_{I_{m - 1}} + k\eps[2] \\
										  &\leq& (1 + \eps\lambda)\kappa(1 + \eps\lambda)\norm{\stapp - \stapp_m}_{I_{m - 1}} + k\eps[2] \\
										  &\leq& \kappa'\norm{\stapp - \st}_{I_{m - 1}} + k(1 + \kappa')\eps[2]
			\earClno
			where $\kappa' = (1 + \eps\lambda)\kappa(1 + \eps\lambda) < 1$ for $0 \leq \eps < \eps_3$.
			
			Then with a simple recursion we obtain
			\beno
				\norm{\stapp - \st}_{I_m} \leq {\kappa'}^{m}\norm{\stapp(\zero{\slt}) - \st(\zero{\slt})} + \Biggl(\sum\limits_{n=0}^{m}{\kappa'}^n\Biggr)(1 + \kappa')k\eps^2.
			\eeno
			Using the fact that $\stapp(\zero{\slt}) = \st(\zero{\slt})$ and that the sum is monotonically increasing, we obtain that
			\beno
				\norm{\stapp - \st}_{I_m} \leq \frac{1 + \kappa'}{1 - \kappa'}k\eps[2].
			\eeno
			Finally, as the previous equation is valid for all $m \in \mathbb{N}$, we find that $\forall \slt \in [\zero{\slt}, \infty[$
			\beno
				\norm{\stapp(\slt)- \st(\slt)} \leq \frac{1 + \kappa'}{1 - \kappa'}k\eps[2] =: C\eps[2].
			\eeno
			Taking $\eps_0 := \min(\{\eps_1, \eps_2, \eps_3\})$, completes the proof.
		\end{proof}
		
		Even though we do not prove it here, we suppose that this result is still valid for higher-order averaging approximations.

\bibliographystyle{phmIEEEtran}
\bibliography{biblio}

\end{document}